\documentclass[a4paper,11pt]{amsart}
\usepackage{amsfonts,amssymb,amsmath,amsthm,abstract,color}
\usepackage[mathscr]{euscript}
\usepackage[ps,all,arc,rotate]{xy}
\usepackage[lmargin=1in,rmargin=1in,tmargin=1in,bmargin=1in]{geometry}
\usepackage{fancyhdr}
\usepackage{pb-diagram}
\usepackage{enumitem}
\usepackage{hyperref}

\newtheorem{prop}{Proposition}[section]
\newtheorem{lemma}[prop]{Lemma}
\newtheorem{thm}[prop]{Theorem}
\newtheorem{cor}[prop]{Corollary}

\theoremstyle{definition}
\newtheorem{defn}[prop]{Definition}

\newtheorem{rmk}[prop]{Remark}
\newtheorem{ex}[prop]{Example}

\DeclareMathOperator{\Hecke}{\cH} 
\DeclareMathOperator{\THecke}{\cF\cH} 
\newcommand{\THH}{\THecke}
\DeclareMathOperator{\Coh}{Coh} 
\newcommand{\TCoh}{\widetilde{\Coh}}
\DeclareMathOperator{\gr}{gr} 
\DeclareMathOperator{\pr}{pr} 
\DeclareMathOperator{\supp}{supp}
\DeclareMathOperator{\PB}{PB}

\DeclareMathOperator{\FDiv}{FDiv}

\DeclareMathOperator{\rk}{rk}        
\DeclareMathOperator{\spec}{Spec} \DeclareMathOperator{\Sym}{Sym}

     \DeclareMathOperator{\Stab}{Stab} 
     
\DeclareMathOperator{\End}{End}
\newcommand{\ra}{\rightarrow}      

\newcommand{\lra}{\longrightarrow}

\DeclareMathOperator{\Bun}{Bun}
 
\DeclareMathOperator*{\hocolim}{hocolim}

\DeclareMathOperator{\Div}{Div} \DeclareMathOperator{\quot}{Quot}
\DeclareMathOperator{\Jac}{Jac} 
\DeclareMathOperator{\Hom}{Hom}
\DeclareMathOperator{\Spec}{Spec}

\DeclareMathOperator{\codim}{codim}

\DeclareMathOperator{\Aut}{Aut}

\newcommand{\et}{\mathrm{\acute{e}t}}

\def\cB{\mathcal B}
\def\cE{\mathcal E}\def\cF{\mathcal F}\def\cH{\mathcal H}
\def\cI{\mathcal I}\def\cL{\mathcal L}
\def\cM{\mathcal M}\def\cO{\mathcal O}
\def\cT{\mathcal T}
\def\cU{\mathcal U}\def\cV{\mathcal V}

\def\AA{\mathbb A}
\def\GG{\mathbb G}

\def\NN{\mathbb N}\def\PP{\mathbb P}
\def\QQ{\mathbb Q}

\def\ZZ{\mathbb Z}

 \def\GL{\mathrm{GL}} 

\def\DM{\mathrm{DM}}  \def\DA{\mathrm{DA}}   
\def\CH{\mathrm{CH}}

\makeatletter
\def\@tocline#1#2#3#4#5#6#7{\relax
  \ifnum #1>\c@tocdepth 
  \else
    \par \addpenalty\@secpenalty\addvspace{#2}%
    \begingroup \hyphenpenalty\@M
    \@ifempty{#4}{%
      \@tempdima\csname r@tocindent\number#1\endcsname\relax
    }{%
      \@tempdima#4\relax
    }%
    \parindent\z@ \leftskip#3\relax \advance\leftskip\@tempdima\relax
    \rightskip\@pnumwidth plus4em \parfillskip-\@pnumwidth
    #5\leavevmode\hskip-\@tempdima
      \ifcase #1
       \or\or \hskip 1em \or \hskip 2em \else \hskip 3em \fi%
      #6\nobreak\relax
    \hfill\hbox to\@pnumwidth{\@tocpagenum{#7}}\par
    \nobreak
    \endgroup
  \fi}
\makeatother

\title[A formula for the motive of the stack of bundles on a curve]{A formula for the Voevodsky motive of the moduli stack of vector bundles on a curve}

\author{Victoria Hoskins and Simon Pepin Lehalleur}

\thanks{V.H. is supported by the DFG Excellence Initiative at the Freie Universit\"{a}t Berlin and by the SPP 1786.}

\begin{document}

\maketitle

\begin{abstract}
We prove a formula for the motive of the stack of vector bundles of fixed rank and degree over a smooth projective curve in Voevodsky's triangulated category of mixed motives with rational coefficients.
\end{abstract}

\tableofcontents

\section{Introduction}

Let $\Bun_{n,d}$ denote the moduli stack of rank $n$, degree $d$ vector bundles on a smooth projective geometrically connected curve $C$ of genus $g$ over a field $k$. In this paper, we prove the following formula for the motive of $\Bun_{n,d}$ in Voevodsky's triangulated category $\DM(k):=\DM(k,\QQ)$ of mixed motives over $k$ with $\QQ$-coefficients.

\begin{thm}\label{main thm}
Suppose that $C(k) \neq \emptyset$; then in $\DM(k,\QQ)$, we have
\[M(\Bun_{n,d}) \simeq  M(\Jac(C)) \otimes M(B\GG_m) \otimes \bigotimes_{i=1}^{n-1} Z(C, \QQ\{i\}),\]
where $Z(C,\QQ\{i\}):=\bigoplus_{j=0}^{\infty} M(C^{(j)})\otimes \QQ\{ij\}$ is a motivic Zeta function and $\QQ\{i\} := \QQ(i)[2i]$.
\end{thm}

In particular, this implies a decomposition on Chow groups and $\ell$-adic cohomology and, as explained below, this formula is compatible with previous cohomological descriptions of $\Bun_{n,d}$. 

This paper is a continuation of our previous work \cite{HPL} in which we define and study the motive $M(\Bun_{n,d}) \in \DM(k,R)$ for any coefficient ring $R$ (provided the characteristic of $k$ is invertible in $R$ in positive characteristic); more generally, we introduce there the notion of an exhaustive stack and define motives of smooth exhaustive stacks by generalising a construction of Totaro for quotient stacks \cite{totaro} (see \cite[Definitions 2.15 and 2.17]{HPL} for details).


\subsection{Overview of our previous results}

In \cite[Theorem 3.5]{HPL}, we work with a general coefficient ring $R$ (for which the exponential characteristic is invertible) and give the following description of the motive of the stack $\Bun_{n,d}$ in terms of smooth projective Quot schemes by following a geometric argument for computing the $\ell$-adic cohomology of this stack in \cite{BGL}.

\begin{thm}\label{thm old main}
For any effective divisor $D>0$ on $C$, we have in $\DM(k,R)$
\[M(\Bun_{n,d})\simeq \hocolim_{l\in\NN} M(\Div_{n,d}(lD)),\]
where $\Div_{n,d}(D)=\{ \cE \subset \cO_C(D)^{\oplus n} : \rk(\cE) =n, \deg(\cE) = d\}$ is a smooth Quot scheme.
\end{thm}

Our approach in \cite{HPL} to describing the motives $M(\Div_{n,d}(lD))$ is to use Bia{\l}ynicki-Birula decompositions \cite{BB_original} associated to an action of a generic one-parameter subgroup $\GG_m\subset \GL_n$ on these Quot schemes, whose fixed loci are disjoint unions of products of symmetric powers of $C$. To use these decompositions to compute the motive of $\Bun_{n,d}$, one needs to understand the behaviour of the transition maps $i_l: \Div_{n,d}(lD) \hookrightarrow \Div_{n,d}((l+1)D)$ in the inductive system in Theorem \ref{thm old main} with respect to the motivic Bia{\l}ynicki-Birula decompositions; this is very complicated, as although the closed immersion $i_l$ is $\GG_m$-equivariant, the closed subscheme $\Div_{n,d}(lD) \hookrightarrow \Div_{n,d}((l+1)D)$ does not intersect the Bia{\l}ynicki-Birula strata transversally. We conjecture a precise description of these transition maps \cite[Conjecture 3.9]{HPL} and show that the formula for the motive of $\Bun_{n,d}$ appearing in Theorem \ref{main thm} follows from this conjectural description of the transition maps.


\subsection{Summary of the results and methods in this paper}

In this paper, we prove the conjectural formula in \cite{HPL} under the assumption that $R = \QQ$. The main idea is to replace the Quot schemes with Flag-Quot schemes, which are generalisations of Quot schemes that allow flags of sheaves and then to describe the transition maps using these Flag-Quot schemes without using Bia{\l}ynicki-Birula decompositions. The idea to use Flag-Quot schemes was inspired by a result of Laumon in \cite{laumon} and its application in a paper of Heinloth to study the cohomology of the moduli space of Higgs bundles using Hecke modification stacks \cite{Heinloth_LaumonBDay}. 

To prove Theorem \ref{main thm}, our starting point is Theorem \ref{thm old main}, where as we assume that $C$ has a rational point $x$ we can take the divisor $D := x$ and write $\Div_{n,d}(l):= \Div_{n,d}(lx)$. We replace the Quot schemes $\Div_{n,d}(l)$ with smooth projective Flag-Quot schemes
\[ \FDiv_{n,d}(l) = \{ \cE_{nl -d} \subsetneq  \cdots \subsetneq \cE_1 \subsetneq \cE_{0} = \cO_C(lx)^{\oplus n} : \rk(\cE_i) = n \text{ and } \deg{\cE_i} = nl-i \}. \]
The natural map $ \FDiv_{n,d}(l) \ra \Div_{n,d}(l)$ is small and is a $S_{nl-d}$-principal bundle over the open subset consisting of subsheaves $\cE \subset \cO_C(lx)^{\oplus n}$ with torsion quotient that has support consisting of $nl-d$ distinct points. Using these facts, we relate the motives of these two varieties as follows.

\begin{thm}\label{thm intro FDiv}
There is an induced $S_{nl-d}$-action on $M(\FDiv_{n,d}(l))$ and 
isomorphisms
\[M(\Div_{n,d}(l)) \simeq M(\FDiv_{n,d}(l))^{S_{nl-d}} \simeq (M(C \times \PP^{n-1})^{\otimes nl-d})^{S_{nl-d}} \simeq \Sym^{nl-d}(M(C \times \PP^{n-1})).\]
\end{thm}

An isomorphism $M(\Div_{n,d}(l)) \simeq \Sym^{nl-d}(M(C \times \PP^{n-1}))$ is constructed by del  Ba\~{n}o \cite[Theorem 4.2]{Del_Bano_motives_moduli} using associated motivic Bia{\l}ynicki-Birula decompositions (see also \cite[$\S$3.2]{HPL}). However, in del  Ba\~{n}o's description, we do not understand the transition maps $M(i_l)$.

In fact, we deduce Theorem \ref{thm intro FDiv} as a special case of a more general result (Theorem \ref{thm Sl action on THecke}), where we replace $\cO_C(lx)^{\oplus n} \ra C/k$ with a family of vector bundles $\cE \ra T \times C/T$ parametrised by a smooth $k$-scheme $T$ and then study the motives of schemes of (iterated) Hecke correspondences as (Flag)-Quot schemes over $T$. This work was inspired by a beautiful description of the cohomology of these schemes due to Heinloth (see the proof of \cite[Proposition 11]{Heinloth_LaumonBDay} which uses ideas of Laumon \cite[Theorem 3.3.1]{laumon}).  In fact we lift Heinloth's cohomological description of schemes of (iterated) Hecke correspondences to $\DM(k)$. To prove this result, in $\S\ref{sec small maps}$ we study the invariant piece of a motive with a finite group action, which is why we need to work with rational coefficients; the main result is Theorem \ref{thm:actions}, which states that for a small proper map $f : X \twoheadrightarrow Y$ of smooth projective $k$-varieties which is a principal $G$-bundle on the locus with finite fibres, we have an isomorphism $M(X)^G \cong M(Y)$. In $\S$\ref{sec motives Hecke schemes}, we study the geometry and motives of schemes of (iterated) Hecke correspondences in order to prove Theorem \ref{thm Sl action on THecke}. Furthermore, we obtain a formula for the motive of the Quot scheme of length $l$ torsion quotients of a rank $n$ locally free sheaf $\cE$ on $C$, which is independent of $\cE$ (Corollary \ref{cor motive torsion quot}); this complements recent analogous results in the Grothendieck ring of varieties \cite{BFP,Ricolfi}.

In $\S$\ref{sec lift trans}, we lift the transition maps $i_l : \Div_{n,d}(l) \ra \Div_{n,d}(l+1)$ to the schemes $\FDiv_{n,d}(l)$. It turns out to be much simpler to describe the motivic behaviour of the lifts of the transition maps to Flag-Quot schemes, as those are iterated projective bundles over products of the curve. By symmetrising this description, we deduce the corresponding behaviour for the maps $M(i_l)$ which enables us to prove Theorem \ref{main thm} in $\S$\ref{sec proof1}. Finally, in $\S$\ref{sec proof2}, we give a second proof of this formula for $M(\Bun_{n,d})$ which follows more closely the ideas in our previous work \cite{HPL}.

It remains an interesting open question as to whether Theorem \ref{main thm} holds integrally. One may expect this to be the case, as Atiyah and Bott \cite{atiyah_bott} gave an integral description of the cohomology of $\Bun_{n,d}$. In fact, in future work we plan to remove the assumption that $C$ has a rational point, by giving a more canonical construction of the isomorphism in Theorem \ref{main thm} inspired by \cite{atiyah_bott}.

By Poincar\'{e} duality, we obtain a formula for the compactly supported motive $M^c(\Bun)$, which compares nicely with previous results, such as the Behrend--Dhillon formula for the virtual class of $\Bun_{n,d}$ in the Grothendieck ring of varieties \cite{BD} and Harder's formula for the stacky point count over a finite field \cite{Harder} (see the discussion in \cite[$\S$4.2]{HPL}).

\begin{cor}
Assume $C(k) \neq \emptyset$; then the compactly supported motive of $\Bun_{n,d}$ is given by
\[M^c(\Bun_{n,d})  \simeq M(\Jac C)\otimes M^c(B\GG_m)\{(n^2 -1 )(g-1) \}   \otimes \bigotimes_{i=2}^n Z(C, \QQ\{-i\}). \]
\end{cor}


\subsection{Background on motives}

Let us briefly recall some basic properties about $\DM(k):=\DM(k,\QQ)$. It is a monoidal $\QQ$-linear triangulated category. For a separated scheme $X$ of finite type over $k$, we can associate a motive $M(X)\in \DM(k)$, which is covariantly functorial in $X$ and behaves like a homology theory. The motive $M(\Spec k):=\QQ\{0\}$ is the unit for the monoidal structure, and there are Tate motives $\QQ\{n\}:=\QQ(n)[2n] \in \DM(k)$ for all $n\in\ZZ$. For any motive $M$ and $n \in \ZZ$, we write $M\{n\}:=M\otimes \QQ\{n\}$. 

In $\DM(k)$, there are K\"{u}nneth isomorphisms, $\AA^1$-homotopy invariance, Gysin distinguished triangles, projective bundle formulae and Poincar\'{e} duality isomorphisms, as well as realisation functors (to compare with Betti, de Rham and $\ell$-adic cohomology) and descriptions of Chow groups as homomorphism groups in $\DM(k)$. For a precise statement of these results, we refer the reader to the summary in \cite[$\S$2]{HPL}.

In this paper, unlike in \cite{HPL}, we need to use categories of relative motives over varying base schemes, and the associated ``six operations'' formalism. We only need a small portion of the machinery, which we summarise here; for more details, see \cite[\S 3]{Ayoub_Survey}. Given a base scheme $S$, which in this paper will always be of finite type and separated over the field $k$, there is a monoidal $\QQ$-linear triangulated category $\DM(S)$, which we take to be the category $\DA^{\et}(S,\QQ)$ of \cite{Ayoub_Survey} and \cite[\S 3]{Ayoub_etale}. The monoidal unit of $\DM(S)$ is denoted by $\QQ_S$ (in particular, $\QQ_k:=\QQ\{0\}\in\DM(k)$). Given a morphism $f:S\ra T$ between two such base schemes (so that $f$ is automatically separated and of finite type), there are two adjunctions
\[ f^*:\DM(T)\leftrightarrows \DM(S):f_* \quad \quad \text{and} \quad \quad f_!:\DM(S)\leftrightarrows \DM(T):f^! \]
which satisfies the same formal properties as the corresponding adjunctions $(f^*,Rf_*)$ and $(Rf_!,f^!)$ in the setting of derived categories of $\ell$-adic sheaves. In particular, we have natural isomorphisms $f_*\simeq f_!$ for $f$ proper, and $f^*\simeq f^!$ for $f$ \'etale. We also have proper base change (in the general form of \cite[Theorem 3.9]{Ayoub_Survey}) and a purity isomorphism $f^!\simeq f^*(-)\{d\}$ for $f$ smooth of relative dimension $d$.

Many constructions in $\DM(k)$ have an alternative description in terms of the six operations formalism:  for a $k$-scheme $X$ with structure map $\pi_X$, we have
\[
M(X)\simeq \pi_{X!}\pi_X^!\QQ_k\quad \quad\text{and} \quad \quad M^c(X)\simeq \pi_{X*}\pi_X^!\QQ_k. \\
\]


\noindent \textbf{Acknowledgements.} We thank Elden Elmanto, Michael Gr\"{o}chenig, Jochen Heinloth, Frances Kirwan and Marc Levine for useful discussions.

\section{Small maps and induced actions on motives}\label{sec small maps}

\subsection{Properties of small maps}

Let us recall the following definition.

\begin{defn}\label{def:small}
  Let $X$ and $Y$ be algebraic varieties\footnote{Here as in the rest of the paper, variety means finite type separated over $k$, not necessarily irreducible.} over $k$, and let $f:X\ra Y$ be a proper morphism. For $\delta\in\NN$, define
  \[
Y_{f,\delta}:=\{y\in Y|\dim(f^{-1}(y))=\delta\}.
\]
This is a locally closed subscheme of $Y$, and so its codimension in $Y$ makes sense. We say $f$ is
\begin{enumerate}[label={\upshape(\roman*)}]
\item semismall if $\codim_{Y}(Y_{f,\delta})\geq 2\delta$ for all $\delta\geq 0$.
\item small if $f$ is semismall and $\codim_{Y}(Y_{f,\delta})> 2\delta$ if for all $\delta> 0$.
\end{enumerate}
\end{defn}

\begin{rmk}
This formulation in terms of codimension also makes sense when $X,Y$ are algebraic stacks and $f:X\ra Y$ is a proper representable morphism.
\end{rmk}

\begin{lemma}\label{lem:small_bc}
Proper (semi)small morphisms are stable by flat base change.
\end{lemma}
\begin{proof}
  Let $f:X\ra Y$ be a proper morphism and $g:Z\ra Y$ be a flat morphism. Write $\tilde{f}:X\times_{Y}Z\ra Z$. With the notations of Definition \ref{def:small}, for all $\delta\in \NN$ we have $Z_{\tilde{f},\delta}=g^{-1}(Y_{f,\delta})$. Since $g$ is flat, we deduce that
  \[
\codim_{Z}(Z_{\tilde{f},\delta}) = \codim_Z (g^{-1}(Y_{f,\delta}))\geq \codim_Y (Y_{f,\delta})
\]
which implies the result.
\end{proof}

\begin{rmk}
This property also holds for proper representable (semi-)small morphisms between algebraic stacks, with the same proof.
\end{rmk}  

The key property of (semi-)small morphisms for this paper is the following lemma.

\begin{lemma}\cite[Proposition 2.1.1, Remark 2.1.2]{dCM_small}\label{lem:small_dim}
Let $f:X\ra Y$ be a proper morphism of varieties. For $\delta\in \NN$, let $Y_{f,\delta}$ be as in Definition \ref{def:small} and $X_{f,\delta}:=f^{-1}(Y_{f,\delta})$. 

\begin{enumerate}[label={\upshape(\roman*)}]
\item \label{dim} The morphism $f$ is semismall if and only if
\[
\dim(X\times_{Y}X)\leq \dim(Y).
\]
\item \label{dim_surj} If $f$ is semismall and surjective, then $\dim(X\times_{Y}X)=\dim(X).$
\item \label{comp} If $f$ is small and surjective, then the irreducible components of dimension $\dim(X)$ of $X\times_{Y}X$ are the closures of the irreducible components of $X_{f,0}\times_{Y_{f,0}}X_{f,0}$ and in particular dominate $Y$.
\end{enumerate}  
\end{lemma}

\subsection{Endomorphisms of motives of small maps}

Given a morphism of schemes $f:X\ra Y$, we denote by $\Aut_Y(X)$ the group of automorphisms of $X$ as a $Y$-scheme. For a $k$-scheme $X$ and an integer $i\in\NN$, we denote by $Z_i(X)$ the group of $i$-dimensional cycles with rational coefficients on $X$, and $\CH_i(X)$ the $i$-th Chow group, i.e., the quotient of $Z_i(X)$ by rational equivalence.

\begin{prop}\label{prop:relative_action}
  Let $f:X\ra Y$ be a proper morphism with $X$ smooth equidimensional of dimension $d\in\NN$. Then there exist an isomorphism
  \[
\phi_{f}:\CH_{d}(X\times_{Y}X)\simeq \End_{\DM(Y)}(f_{*}\QQ_{X})
\]
such that, if $e:U\hookrightarrow Y$ is a \'etale morphism and $\tilde{e}:V\hookrightarrow X$ is its base change along $f$ and $\tilde{f}:V\ra U$ the base change of $f$ along $e$, we have a commutative diagram
  \[
    \xymatrix{
      \CH_d(X\times_Y X) \ar[d]^{(\tilde{e}\times\tilde{e})^*} \ar[r]_{\phi_f} & \End_{\DM(Y)}(f_{*}\QQ_X) \ar[d]^{e^*} \\
      \CH_d(V\times_U V) \ar[r]_{\phi_{\tilde{f}}} & \End_{\DM(U)}(\tilde{f}_{*}\QQ_{V}).
    }
  \]

\end{prop}

\begin{proof}
Write $p_1,p_2:X\times_Y X\ra X$ for the two projections. For a $k$-scheme $Z$, write $\pi_Z:Z\ra \Spec(k)$ for its structure map. We start with the isomorphism 
  \[
\CH_d(X\times_Y X)\simeq \Hom_{\DM(k)}(\QQ\{d\},M^c(X\times_Y X))\simeq \Hom_{\DM(X\times_Y X)}(\QQ_{X\times_Y X},\pi^!_{X\times_Y X}\QQ\{-d\})
\]
where we have used the description of Chow groups for general varieties in $\DM(k)$, the formula for $M^c$ in terms of the six operations and the adjunction $(\pi_{X\times_Y X}^*,\pi_{X\times_Y X *})$. We then write
  \begin{flalign*}
    \Hom_{\DM(X\times_Y X)}(\QQ_{X\times_Y X},\pi^!_{X\times_Y X}\QQ_k\{-d\}) & \simeq  \Hom_{\DM(X\times_Y X)}(\QQ_{X\times_Y X},p_1^!\pi_{X}^{!}\QQ_k\{-d\}) \\
                                                                       & \simeq  \Hom_{\DM(X\times_Y X)}(\QQ_{X\times_Y X}, p_1^!\QQ_X) \\
                                                                       & \simeq  \Hom_{\DM(X)}(\QQ_X,p_{1*}p_1^!\QQ_{X}) \\
                                                                       & \simeq  \Hom_{\DM(X)}(\QQ_{X},f^{!}f_{*}\QQ_{X}) \\
    & \simeq  \End_{\DM(Y)}(f_{*}\QQ_{X})
\end{flalign*}    
where the first isomorphism follows from $\pi_{X\times_{Y}X}=\pi_{X} \circ p_{1}$, the second follows from relative purity for the smooth morphism $\pi_{X}$, the third is the adjunction $(p_{1}^{*},p_{1*})$, the fourth is proper base change and the fifth uses the adjunction $(f_{!},f^{!})$ and the properness of $f$.

The isomorphism $\phi_f$ is defined as the composition of the sequence of isomorphisms above. Its compatibility with pullback by an \'etale morphism $e$ is a matter of carefully going through the construction and using the natural isomorphism $e^!\simeq e^*$ and proper base change.
\end{proof}

\begin{rmk}
Since the target of $\phi_f$ is clearly a $\QQ$-algebra, the proposition endows $\CH_d(X\times_Y X)$ with a $\QQ$-algebra structure. The multiplication can be described using refined Gysin morphisms, but we will not need this.
\end{rmk}  

\begin{prop}\label{prop:end_ring_iso}
  Let $f:X\ra Y$ be a surjective proper small morphism with $X$ and $Y$ smooth varieties. Let $f^\circ:X^\circ\ra Y^\circ$ be the restriction of $f$ to the locus with finite fibers, and $j:Y^\circ\ra Y$ the corresponding open immersion. Then the natural map
  \[
j^*:\End_{\DM(Y)}(f_!f^!\QQ_Y)\ra \End_{\DM(Y^\circ)}(f^\circ_!f^{\circ !}\QQ_{Y^\circ})
 \]
is an isomorphism of rings.
\end{prop}

\begin{proof}
First, let us explain how $j^*$ is defined. Write $\tilde{\jmath}:X^\circ\ra X$. Then we have
\[j^*f_!f^!\QQ_Y\simeq f^\circ_!\tilde{\jmath}^*f^!\QQ_Y\simeq f^\circ_!\tilde{\jmath}^!f^!\QQ_Y\simeq f^\circ_!(f^\circ)^!j^!\QQ_Y\simeq f^\circ_!(f^\circ)^!j^*\QQ_Y\simeq f^\circ_!(f^\circ)^!\QQ_{Y^\circ}\]
where we have used proper base change, compatibility of $(-)^{!}$ with composition and the fact that $e^!\simeq e^*$ for $e$ \'etale. Then $j^*$ is defined as
\[
\End_{\DM(Y)}(f_!f^!\QQ_Y)\stackrel{j^*}{\ra}\End_{\DM(Y)}(j^*f_!f^!\QQ_Y) \simeq \End_{\DM(Y^\circ)}(f^\circ_!f^{\circ !}\QQ_{Y^\circ}).
\]
The map $j^*$ is clearly compatible with addition and composition, hence is a homomorphism of rings. It remains to show that it is bijective.

Since $X$ and $Y$ are both smooth of dimension $d$ over $k$, we can use purity isomorphisms to obtain an isomorphism
\begin{equation}
\label{f_purity}
f^{!}\QQ_{Y}\simeq f^{!}\pi_{Y}^{!}\QQ_{k}\{-d\}\simeq \pi_{X}^{!}\QQ_{k}\{-d\}\simeq \QQ_{X}.
\end{equation}
We deduce that $f_*\QQ_X\simeq f_*f^!\QQ_Y\simeq f_!f^!\QQ_Y$, and similarly that $f^\circ_*\QQ_{X^\circ}\simeq f^\circ_!f^{\circ !}\QQ_{Y^\circ}$. These two isomorphisms are compatible with restriction along $j$. Combining this observation with Proposition \ref{prop:relative_action}, we have the commutative diagram with horizontal isomorphisms

\[
\xymatrix{  
\CH_d(X\times_{Y} X) \ar[d]_{(j\times j)^*} \ar[r]^{\sim}_{\phi_{f}} & \End_{\DM(Y)}(f_*\QQ_X) \ar[d]_{j^*} \ar[r]^{\sim} & \End_{\DM(Y)}(f_!f^{ !}\QQ_{Y}) \ar[d]_{j^*} \\    
\CH_d(X^\circ \times_{Y^\circ} X^\circ) \ar[r]^{\sim}_{\phi_{f^\circ}} & \End_{\DM(Y^\circ)}(f^\circ_*\QQ_{X^\circ}) \ar[r]^{\sim} & \End_{\DM(Y^\circ)}(f^\circ_!f^{\circ !}\QQ_{Y^\circ}).
}
  \]
On a variety of dimension $d$, we have $\CH_d=Z_d$, i.e., rational equivalence is trivial on top-dimensional cycles. By Lemma \ref{lem:small_dim} \ref{dim_surj}, this implies $\CH_d(X\times_Y X)\simeq Z_d(X\times_Y X)$ and also $\CH_d(X\times_Y X)\simeq Z_d(X^\circ\times_{Y^\circ}X^\circ)$. By Lemma \ref{lem:small_dim} \ref{comp}, the restriction morphism $Z_d(X\times_Y X)\ra Z_d(X^\circ\times_{Y^\circ} X^\circ)$ is a bijection. We deduce that the left vertical map in the diagram above is a bijection, and conclude that the right vertical map is a bijection.
\end{proof}  

\begin{lemma}\label{lem:psi}
Let $f:X\ra Y$ be a finite type separated morphism with $Y$ smooth. Then there exist an morphism of $\QQ$-algebras
\[
  \psi_f:\End_{\DM(Y)}(f_!f^!\QQ_Y)\ra \End_{\DM(k)}(M(X))
\]
such that, for $e:U\hookrightarrow Y$ an \'etale morphism, $\tilde{e}:V\hookrightarrow X$ its base change along $f$ and $\tilde{f}:V\ra U$ the base change of $f$ along $e$, we have a commutative diagram
  \[
    \xymatrix{
\End_{\DM(Y)}(f_!f^!\QQ_Y) \ar[d]^{e^*} \ar[r]_{\psi_f} & \End_{\DM(k)}(M(X)) \ar[d]^{e^*} \\
\End_{\DM(U)}(\tilde{f}_!\tilde{f}^!\QQ_U) \ar[r]_{\psi_{\tilde{f}}} & \End_{\DM(k)}(M(V)).
    }
  \]
\end{lemma}
\begin{proof}
Recall that, for $Z$ a smooth variety of dimension $e$ over $k$, we have a canonical purity isomorphism $\pi_{Z}^{!}\QQ_{k}\simeq \QQ_{Z}\{e\}$. By working with each connected component of $Y$ separately, we can assume that $Y$ is equidimensional of dimension $d$. We deduce that
  \[
M(X):= \pi_{X!}\pi^{!}_{X}\QQ_{k}\simeq \pi_{Y!}f_{!}f^{!}\pi_{Y}^{!}\QQ_{k}\simeq \pi_{Y!}f_{!}f^!\QQ_{Y}\{d\}
\]
by using the purity isomorphism for the smooth morphism $\pi_Y$. We define $\psi_f$ as the composition
\[\End_{\DM(Y)}(f_!f^!\QQ_Y)\stackrel{\pi_{Y!}(-)\{d\}}{\lra} \End_{\DM(k)}(\pi_{Y!}f_!f^!\QQ_Y\{d\})\simeq \End_{\DM(k)}M(X).\]
The compatibility with pullbacks by \'etale morphisms follows again easily from the natural isomorphism $e^!\simeq e^*$ for an \'etale morphism $e$.
\end{proof}


\subsection{Group actions on motives of small maps} 

Let $S$ be a scheme, $M\in\DM(S)$ a motive and $G$ a group. An action of $G$ on $M$ is a morphism of groups $a:G\ra \Aut_{\DM(S)}(M)$.
In particular, given a morphism $f:X\ra Y$, we have an action $\Aut_Y(X)\ra \Aut_{\DM(Y)}(f_!f^!\QQ_Y)$.

Assuming further that $G$ is finite, let
\[
\Pi_{a}:=\frac{1}{|G|}\sum_{g\in G}a(g)\in \End_{\DM(S)}(M)
  \]
  which makes sense since $\DM(S)$ is $\QQ$-linear. Then $\Pi_{a}$ is idempotent, and since $\DM(S)$ is idempotent-complete we define the invariant motive $M^{G}\in\DM(S)$ as the image of $\Pi_{a}$.

\begin{ex}\label{ex:sym}
An important example for this paper are motives of symmetric products. For a quasi-projective variety $X$ over $k$ and $n\in\NN$, we have a morphism $f:X^n\ra \Sym^n(X)$. The symmetric group $S_n$ acts on $X^n$ over $\Sym^n(X)$, so that we get an induced action on $M(X^n)$ such that $M(f):M(X^n)\ra M(\Sym^n(X))$ factors via $M(X^n)^{S_n}\ra M(\Sym^n(X))$. Since $S_n$ acts transitively on the geometric fibers of $f$, this second morphism is an isomorphism $M(X^n)^{S_n}\simeq M(\Sym^n(X))$ by \cite[Corollaire 2.1.166]{Ayoub_these_1}.
\end{ex}

The main result of this section is a generalisation of the previous example where we do not have a global action on $X$ and $f$ is not necessarily finite but only small.
  


\begin{thm}\label{thm:actions}
  Let $f:X\ra Y$ be a small surjective proper morphism between smooth connected varieties. Assume that the restriction $f^\circ:X^\circ\ra Y^\circ$ to the locus with finite fibers is a principal $G$-bundle. Then the action of $G$ on $M(X^\circ)$ extends to an action on $M(X)$ which induces an isomorphism $M(X)^{G}\simeq M(Y)$; we have a commutative diagram
  \[
    \xymatrix{
M(X^\circ) \ar[r] \ar[d] & M(X^\circ)^{G} \ar[r]^{\sim} \ar[d] & M(Y^\circ)\ar[d] \\
M(X) \ar[r] & M(X)^{G} \ar[r]^{\sim} & M(Y). \\
    }
    \]
\end{thm}

\begin{proof}
  By working separately with each connected component, we can assume $Y$ is connected, and in particular equidimensional. Write $d=\dim(X)=\dim(Y)$. Since $f^\circ:X^\circ\ra Y^\circ$ is a principal $G$-bundle, we have a morphism of groups $G\ra \Aut_{Y^\circ}(X^\circ)$. We deduce a morphism of groups $G\ra \Aut_{\DM(Y^\circ)}(f^\circ_!f^{\circ !}\QQ_{Y^\circ})$. By Proposition \ref{prop:end_ring_iso}, this yields a morphism of groups $G\ra \Aut_{\DM(Y)}(f_!f^!\QQ_Y)$. We compose with the morphism $\psi_f$ of Lemma \ref{lem:psi} and get a morphism of groups $G\ra\Aut_{\DM(k)}(M(X))$, which is the required action.

  Let us check that the morphism $M(f):M(X)\ra M(Y)$ factors through $M(X)^G$. Given its construction, it suffices to show that the counit morphism $f_!f^!\QQ_Y\ra \QQ_Y$ factors through $(f_!f^!\QQ_Y)^G$. For this, it suffices to show that, for any $g\in G$, the composition $f_!f^!\QQ_Y\stackrel{g}{\ra}f_!f^!\QQ_Y\ra \QQ_Y$ coincides with the counit of the adjunction $(f_!,f^!)$. By the same adjunction, this amounts to comparing two maps $f^!\QQ_Y\ra f^!\QQ_Y$. By equation \eqref{f_purity}, we have $f^!\QQ_Y\simeq \QQ_X$. By \cite[Proposition 11.1]{Ayoub_etale}, and using the fact that $X^\circ$ is dense in $X$, we have
  \[
\Hom_{\DM(X)}(\QQ_X,\QQ_X)\simeq \QQ^{\pi_0(X)}\hookrightarrow \QQ^{\pi_0(X^{\circ})}\simeq \Hom_{\DM(X^\circ)}(\QQ_{X^\circ},\QQ_{X^\circ})
  \]
hence we can check the required equality after restriction to $X^\circ$; that is, we must show that for any $g\in G$, the composition $f^\circ_!f^{\circ !}\QQ_{Y^\circ}\stackrel{g}{\ra}f^\circ_!f^{\circ !}\QQ_{Y^\circ}\ra \QQ_{Y^\circ}$ coincides with the counit of the adjunction $(f^\circ_!,f^{\circ !})$. This is clear since $G$ acts through $\Aut_{Y^\circ}(X^\circ)$. 

By construction, to show that the induced map $M(X)^G\ra M(Y)$ is an isomorphism, it suffices to show that the morphism $(f_! f^!\QQ_Y)^G\ra \QQ_Y$ is an isomorphism. Let $\Pi_G\in \End_{\DM(Y)}(f_!f^!\QQ_Y)$ the projector onto $(f_! f^!\QQ_Y)^G$. Since $X$ and $Y$ are smooth of the same dimension $d$, the purity isomorphisms yield an isomorphism $f^!\QQ_Y\simeq \QQ_X$ (equation \eqref{f_purity}). Moreover, this isomorphism is compatible with restriction to $Y^\circ$, in the sense that after applying $\tilde{\jmath}^!=\tilde{\jmath}^*$ for $\tilde{\jmath}:X^{\circ}\ra X$, it coincides with the simpler isomorphism $f^{\circ !}\QQ_{Y^\circ}\simeq f^{\circ *}\QQ_{Y^\circ}\simeq \QQ_{X^\circ}$ (using that $f^\circ$ is \'etale).

Consider the composition
\[
\Pi':f_!f^!\QQ_Y\stackrel{\eta_!}{\ra} \QQ_Y\stackrel{\frac{1}{|G|}}{\ra}\QQ_Y\stackrel{\epsilon_*}{\ra} f_*\QQ_X\simeq f_!f^!\QQ_Y
\]
where $\epsilon_*$ is the unit for the adjunction $(f^*,f_*)$ and $\eta_!$ is the counit for the adjunction $(f_!,f^!)$. By \cite[Lemme 2.1.165]{Ayoub_these_1}, we see that $j^{*}\Pi'$ is a projector which coincides with $j^{*}\Pi_{G}$. By the injectivity of $j^{*}$ (Proposition \ref{prop:end_ring_iso}), this implies that $\Pi'=\Pi_G$, thus $\Pi'$ is a projector, and to conclude it remains to identify the image of $\Pi'$ with the morphism $f_!f^!\QQ_Y\ra \QQ_Y$.

For this, it is clearly enough to show that the composition
\[
\QQ_Y\stackrel{\epsilon_*}{\ra} f_*\QQ_X\simeq f_!f^!\QQ_Y\stackrel{\eta_!}{\ra} \QQ_Y
\]
coincides with the multiplication by $|G|$. Since $Y$ and $Y^\circ$ are connected, by \cite[Proposition 11.1]{Ayoub_etale} we have
\[\Hom_{\DM(Y)}(\QQ_Y,\QQ_Y)\simeq \QQ\simeq \Hom_{\DM(Y^\circ)}(\QQ_{Y^\circ},\QQ_{Y^\circ})\]
hence it is enough to show this after restriction to $Y^\circ$. The corresponding composition is
\[
\QQ_{Y^\circ}\stackrel{\epsilon_*}{\ra} f^\circ_*\QQ_{X^\circ}\simeq f^\circ_!f^{\circ !}\QQ_{Y^\circ}\stackrel{\eta_!}{\ra} \QQ_{Y^\circ}
\]
which coincides with multiplication by $|G|$ by \cite[Lemme 2.1.165]{Ayoub_these_1}. 
\end{proof}

\begin{rmk}\label{rmk:functoriality}
  Consider a commutative diagram
  \[
    \xymatrix{
      X \ar[r]^{h} \ar[d]_{f} & X' \ar[d]^{f'} \\
      Y \ar[r]^{g} & Y'
    }
  \]
with $f$ and $f'$ satisfying the assumptions of Theorem \ref{thm:actions} with groups $G$, $G'$. If $g$ does not send the locus $Y^0$ into $(Y^0)'$, it is not clear how to formulate conditions which make the morphism $M(X)\ra M(X')$ equivariant with respect to some given homomorphism $G\ra G'$. However in the application in $\S$\ref{sec formula}, we have an alternative description of the actions which make a certain equivariance property clear (see Proposition \ref{prop THecke trans}).
\end{rmk}

\section{Motives of schemes of Hecke correspondences}\label{sec motives Hecke schemes}

In this section, we introduce some generalisations of the schemes of matrix divisors $\Div_{n,d}(D)$ and the flag-generalisation $\FDiv_{n,d}(D)$ and study their motives. The main result in this section is inspired by work of Laumon \cite{laumon} and Heinloth \cite{Heinloth_LaumonBDay}.

\subsection{Definitions and basic properties}

For a family $\cE$ of vector bundles on $C$ parametrised by a $k$-scheme $T$, we write $\rk(\cE) = n$ and $\deg(\cE) = d$ if the fibrewise rank and degree of this family are $n$ and $d$ respectively.

\begin{defn}
For  $l \in \NN$ and a family $\cE$ of rank $n$ degree $d$ vector bundles over $C$ parametrised by $k$-scheme $T$, we define two $T$-schemes $\Hecke^l_{\cE/T}$ and $\THecke^l_{\cE/T}$ as follows: over $g: S \ra T$, the points of these schemes are given by
\[ \Hecke^l_{\cE/T} (S):= \left\{\phi : \cF \hookrightarrow (g \times \mathrm{id}_C)^*\cE : \begin{array}{c} \cF \ra S \times C \text{ family of vector bundles on } C  \\  \rk(\cF) = n, \deg(\cF)=d-l, \rk(\phi)=n  \end{array} \right\} \]
and 
\[ \THecke^l_{\cE/T} (S):= \left\{ \cF_l \hookrightarrow \cF_{l-1}  \cdots  \hookrightarrow \cF_0 := (g \times \mathrm{id}_C)^*\cE :  \begin{array}{c} \cF_i \ra S \times C \text{ family of vector bundles}  \\  \rk(\cF_i) = n, \deg(\cF_i)=d-i \\ \rk(\cF_i \ra \cF_{i-1})=n \text{ for } i = 1, \cdots, l \end{array}  \right\}. \]
We refer to $\Hecke^l_{\cE/T}$ as the $T$-scheme of length $l$ Hecke correspondences of $\cE$ and the $\THecke^l_{\cE/T}$ as the $T$-scheme of $l$-iterated Hecke correspondences of $\cE$. 
\end{defn}

Let us first explain why these are both schemes over $T$. The scheme of length $l$ Hecke correspondences $\Hecke^l_{\cE/T}$ is the Quot scheme over $T$
\[\Hecke^l_{\cE/T} = \quot_{T \times C /T}^{(0,l)}(\cE)\]
parametrising quotients families of $\cE$ of rank $0$ and degree $l$, which is a projective $T$-scheme. Similarly $\THecke^l_{\cE/T}$ is a generalisation of Quot schemes to allow flags of arbitrary length, called a Flag-Quot or Drap scheme (see \cite[Appendix 2A]{HL}); thus $\THecke^l_{\cE/T}$ is also projective over $T$. In fact, as we are considering torsion quotients of a smooth projective curve, both $\Hecke^l_{\cE/T}$ and $\THecke^l_{\cE/T}$ are smooth $T$-schemes (see \cite[Propositions 2.2.8 and 2.A.12]{HL}). In particular, if $T/k$ is smooth (resp. projective), then both these schemes are smooth (resp.) projective over $k$.

\begin{ex}
Let $T = \spec (k)$ and $\cE = \cO_C(D)^{\oplus n}$ for a divisor $D$ on $C$; then
\[ \Hecke^{n\deg(D) - d}_{\cE/T} = \Div_{n,d}(D) \quad \text{and} \quad \THecke^{n \deg(D) -d}_{\cE/T} = \FDiv_{n,d}(D), \]
which are both smooth and projective.
\end{ex}

We introduce some notation and properties of these Hecke schemes in the following remark.

\begin{rmk}\label{rmk structure of Hecke schemes}
Let $\cE $ be a family of rank $n$ degree $d$ vector bundles over $C$ parametrised by $T$.
\begin{enumerate}[label={\upshape(\roman*)}]
\item\label{rmk Hecke 1} For $l=0$, we note that $\THecke^0_{\cE/T} = \Hecke^0_{\cE/T} = T$ and for $l = 1$, we have
\[ \THecke^1_{\cE/T} = \Hecke^1_{\cE/T} \cong \PP(\cE) \ra T \times C, \]
where this projection is given by taking the support of the family of degree 1 torsion sheaves. Indeed, an elementary modification of a vector bundle $E \ra C$ at $x \in C$ is equivalent to a surjection $E_x \twoheadrightarrow \kappa(x)$ (up to scalar multiplication). 
\item\label{rmk Hecke 2} Since $\THecke^l_{\cE/T}$ is a Flag-Quot scheme there is a universal flag of vector bundles
\[ \cU^l_l \hookrightarrow \cU^l_{l-1} \hookrightarrow   \cdots \hookrightarrow  \cU^l_{1} \hookrightarrow \cU^l_0 := p_2^*\cE \]
over $\THecke^l_{\cE/T} \times_T (T \times C) \cong \THecke^l_{\cE/T} \times C$. 
In fact, Flag-Quot schemes, and in particular schemes of iterated Hecke correspondences, are constructed as iterated relative Quot schemes. More precisely, we have 
\[ \pi_l: \THecke^l_{\cE/T} \cong \Hecke^1_{\cU^{l-1}_{l-1}/\THecke^{l-1}_{\cE/T}} \cong \PP(\cU^{l-1}_{l-1}) \ra \THecke^{l-1}_{\cE/T} \times_T (T \times C) \cong \THecke^{l-1}_{\cE/T} \times C. \]
where $\pi_l(\cF_l \subsetneq \cF_{l-1} \subsetneq \cdots \subsetneq \cF_0) := (\cF_{l-1} \subsetneq \cdots \subsetneq \cF_0, \supp(\cF_{l-1}/\cF_l))$
\item\label{rmk Hecke 3} There is a map $P_l :  \THecke^l_{\cE/T}  \ra T \times C^l$ obtained by composing the maps $\pi_j$ for $ 1 \leq j \leq l$
\[ \quad \quad \quad P_{l} :  \THecke^l_{\cE/T} \ra \THecke^{l-1}_{\cE/T} \times_T (T \times C) \ra \THecke^{l-2}_{\cE/T} \times_T (T \times C)^{\times_T \: 2} \cdots \ra T \times_T (T \times C)^{\times_T \: l} .\]
Explicitly, we have $P_l(\cF_l \subsetneq \cF_{l-1} \subsetneq \cdots \subsetneq \cF_0) = (\supp(\cF_{0}/\cF_1),\ldots , \supp(\cF_{l-1}/\cF_l))$. 
\item\label{rmk Hecke 4} For $1 \leq j \leq l$, we let $\pr_{j}^l :  \THecke^l_{\cE/T} \ra T \times C$ denote the composition of $P_{l}$ with the projection onto the $j$th copy of $T \times C$; that is,
\[ \pr_{j}^l(\cF_l \subsetneq \cF_{l-1} \subsetneq \cdots \subsetneq \cF_0) = \supp(\cF_{j-1}/\cF_j).\]
\item\label{rmk Hecke 5} Let $p_l : \THecke^l_{\cE/T}  \ra \THecke^{l-1}_{\cE/T} $ denote the composition of $\pi_l$ with the projection to the first factor; then for $1 \leq j \leq l-1$, we have $(p_l \times \mathrm{id}_C)^*\cU_{j}^{l-1} = \cU_j^l$.
\end{enumerate}
\end{rmk}

\begin{lemma}\label{lemma it proj bdle}
Let $\cE$ be a family of rank $n$ degree $d$ vector bundles over $C$ parametrised by a scheme $T$; then the scheme $\THecke^l_{\cE/T}$ is an $l$-iterated $\PP^{n-1}$-bundle over $T \times C^l$. More precisely, we have the following sequence of projective bundles
\[ \THecke^l_{\cE/T} \cong \PP(\cU^{l-1}_{l-1}) \ra \THecke^{l-1}_{\cE/T} \times C \cong \PP(\cU^{l-2}_{l-2}) \times C \ra \cdots \ra \THecke^1_{\cE/T}\times C^{l-1} \cong \PP(\cE) \times C^{l-1} \ra T \times C^l.  \]
\end{lemma}
\begin{proof}
This follows by induction from Remark \ref{rmk structure of Hecke schemes} \ref{rmk Hecke 1} and \ref{rmk Hecke 2}.
\end{proof}

By repeatedly applying the projective bundle formula, we obtain the following corollary.

\begin{cor}\label{cor motive THecke}
Let $\cE$ be family of rank $n$ degree $d$ vector bundles over $C$ parametrised by a scheme $T$. Then
\[ M(\THecke^l_{\cE/T}) \cong M(T) \otimes M(C \times \PP^{n-1})^{\otimes l}. \]
\end{cor}

In fact, we will need to explicitly identify this isomorphism. For a rank $n$ vector bundle $\cV$ over a scheme $X$, the projective bundle $\pi : \PP(\cV) \ra X$ is equipped with a line bundle $\cL:=\cO_{\PP(\cV)}(1)$. The first chern class of this line bundle defines a map $c_1(\cL) : M(\PP(\cV)) \ra \QQ\{1\}$ and for $i \geq 0$ it induces maps
\[ c_1(\cL)^{\otimes i} : M(\PP(\cV)) \stackrel{M(\Delta)}{\lra} M(\PP(\cV))^{\otimes i}  \stackrel{ c_1(\cL)^{\otimes i} }{\lra} \QQ\{i\}\]
which together define a map $[c_1(\cL)]:=\oplus_{i=0}^{n-1} c_1(\cL)^{\otimes i} : M(\PP(\cV)) \ra \oplus_{i=-0}^{n-1}\QQ\{i\} \simeq M(\PP^{n-1})$. Then the projective bundle formula isomorphism can be explicitly written as the composition
\[ \PB(\cL) : M(\PP(\cV)) \stackrel{M(\Delta)}{\lra} M(\PP(\cV))^{\otimes 2} \stackrel{M(\pi) \otimes [c_1(\cL)] }{\xrightarrow{\hspace*{1cm}}} M(X) \otimes M(\PP^{n-1}). \]

\begin{rmk}\label{rmk line bundles THecke} 
On $\THH^l:=\THecke^l_{\cE/T}$, we can inductively define $l$ line bundles $\cL_1^1, \cdots, \cL_1^l$ by
\begin{enumerate}[label={\upshape(\roman*)}]
\item $\cL^l_l:=\cO(1) \ra \PP(\cU_{l-1}^{l-1})$,
\item $\cL^l_j:= p_l^* \cL^{l-1}_j$ for $ 1 \leq j \leq l-1$, where $p_l : \THH^l \ra \THH^{l-1}$.
\end{enumerate}
These $l$ line bundles on $\THH^l$ induce a morphism
\[ \PB(\cL_\bullet^l) :  M(\THecke^l_{\cE/T}) \stackrel{M(\Delta)}{\lra} M(\THecke^l_{\cE/T})^{\otimes l+1}  \stackrel{M(P_l) \otimes [c_1(\cL^l_\bullet)]}{\lra} M(T \times C^l) \otimes M( \PP^{n-1})^{\otimes l}, \]
where $ [c_1(\cL^l_\bullet)] = \otimes_{i=1}^l  [c_1(\cL^l_i)]$. Furthermore,  on $\THH^l$ we have two universal objects:
\begin{enumerate}[label={\upshape(\roman*)}]
\item a surjection $\pi_l^*\cU^{l-1}_{l-1} \twoheadrightarrow \cL^l_l$ over $\THH^l$ (as $\THH^l\cong \PP(\cU^{l-1}_{l-1})$ by Remark \ref{rmk line bundles THecke}),
\item a short exact sequence $0 \ra \cU_l^l \ra \cU^l_{l-1} \ra \cT_l^l \ra 0$ over $\THH^l \times C$.
\end{enumerate}
Since $ \cU^l_{l-1} = (p_l \times \mathrm{id}_C)^*\cU^{l-1}_{l-1}$, the relationship between the line bundle $\cL^l_l \ra \THH^l$ and the family of degree 1 torsion sheaves $\cT^l_l$ on $C$ parametrised by $\THH^l$ is 
\[\cL^l_l \cong ((\mathrm{id}_{\THH^l} \times \pi_l) \circ \Delta_{\THH^l})^*\cT_l^l,\] 
for $(\mathrm{id}_{\THH^l} \times \pi_l) \circ \Delta_{\THH^l}: \THH^l \stackrel{\Delta_{\THH^l}}{\lra}  \THH^l \times_{\THH^{l-1}}  \THH^l \stackrel{\mathrm{id} \times \pi_l}{\lra} \THH^l \times_{\THH^{l-1}} (\THH^{l-1} \times C) \simeq \THH^l \times C$. In fact, for $ 1 \leq j \leq l$, we can define maps 
\[ r^l_j = (\mathrm{id}_{\THH^l} \times \pr^l_j) \circ \Delta_{\THH^l} : \THH^l \ra \THH^l \times_T \THH^l \ra  {\THH^l} \times_T (T \times C) \cong {\THH^l} \times C \]
such that $r_l^l = (\mathrm{id}_{\THH^l} \times \pi_l) \circ \Delta_{\THH^l}$. For $j < l$, the family of degree 1 torsion sheaves $\cT^l_j:= \cU^l_{j-1}/\cU^l_{j}$ on $C$ parametrised by $\THH^l$ is obtained as a pullback of $\cT^{l-1}_j$ via the map $p_l \times \mathrm{id}_C$. Hence, for $ 1 \leq j \leq l$, we have isomorphisms relating the line bundles and families of torsion sheaves
\begin{equation}\label{eq line bundles and torsion quotients}
\cL^l_j \cong (r^l_j)^*\cT^l_j.
\end{equation}
\end{rmk}

We can now give a precise description of the isomorphism in Corollary \ref{cor motive THecke}.

\begin{lemma}\label{lemma PB for THecke with line bundles}
The tuple $\cL_\bullet^l = (\cL^l_1, \cdots , \cL^l_l)$ of line bundles on $\THecke^l_{\cE/T}$ induces a morphism
\[ \PB(\cL_\bullet^l) :  M(\THecke^l_{\cE/T}) \stackrel{M(\Delta)}{\lra} M(\THecke^l_{\cE/T})^{\otimes l+1}  \stackrel{M(P_l) \otimes [c_1(\cL^l_\bullet)]}{\xrightarrow{\hspace*{1cm}}} M(T \times C^l) \otimes M( \PP^{n-1})^{\otimes l}, \]
which coincides with the composition
\[ M(\THecke^l_{\cE/T}) \stackrel{\PB(\cL_l^l)}{\lra} M(\THecke^{l-1}_{\cE/T}) \otimes M( C \times \PP^{n-1}) \stackrel{\PB(\cL_{l-1}^{l-1}) \otimes M(\mathrm{id})}{\lra} \cdots \lra M(T \times C^l) \otimes M( \PP^{n-1})^{\otimes l}\]
and thus is an isomorphism.
\end{lemma}
\begin{proof}  For this one uses that Chern classes are compatible with pullbacks, so that $ c_1(\cL^{l-1}_j) \circ M(p_l)  = c_1(\cL^l_{j})$ for $1 \leq j \leq l-1$, as $p_l^*(\cL_j^{l-1})=\cL_j^l$. Then one uses that $P_l$ is defined as the composition of the maps $\pi_i$ for $i \leq l$ together with the fact that for any morphism $f : X \ra Y_1 \times Y_2$, we have the following commutative diagram
\[ \xymatrix{  M(X) \ar[r]^{M(\Delta)\quad\quad} \ar[d]_{M(f)} & M(X) \otimes M(X) \ar[d]^{M(f_1) \otimes M(f_2)}\\
 M(Y_1 \times Y_2) \ar[r]^{\simeq \quad} & M(Y_1) \otimes M(Y_2),}\]
where $f_i := \mathrm{pr}_i \circ f : X \ra Y_i$ and the lower map in this square is the K\"{u}nneth isomorphism.
\end{proof}

\subsection{The motive of the scheme of Hecke correspondences}

There is a forgetful map
\[ f : \THecke^l_{\cE/T}  \ra \Hecke^l_{\cE/T} \]
that we will use to relate the motive of $\Hecke^l_{\cE/T}$ to that of $\THecke^l_{\cE/T} $, which we computed above. In fact, we plan to use the above section to compare these motives, as the map $f$ is small. To prove that $f$ is a small map, we will describe it as the pullback of a small map along a flat morphism by generalising an argument of Heinloth \cite[Proposition 11]{Heinloth_LaumonBDay}.

Let $\Coh_{0,l}$ denote the stack of rank $0$ degree $l$ coherent sheaves on $C$ and let $\TCoh_{0,l}$ denote the stack which associates to a scheme $S$ the groupoid
\[ \TCoh_{0,l} (S) = \langle \cT_1 \hookrightarrow \cT_2  \hookrightarrow \cdots \hookrightarrow \cT_l : \cT_i \in \Coh_{0,i}(S)  \rangle.\]
The forgetful map $f' : \TCoh_{0,l} \ra \Coh_{0,l}$ fits into the following commutative diagram
 \begin{equation}\label{diag p q}
 \xymatrix{  \THecke^l_{\cE/T} \ar[d]^{f} \ar[r]^{\tilde{\gr}} & T \times \TCoh_{0,l} \ar[d]^{\mathrm{id}_T \times f'} \ar[r] & T \times C^l \ar[d]\\
\Hecke^l_{\cE/T} \ar[r]^{{\gr}} & T \times \Coh_{0,l} \ar[r] & T \times C^{(l)} }
 \end{equation}
such that the left square in this diagram is Cartesian. Furthermore, by \cite[Theorem 3.3.1]{laumon}, the map $f'$ is small and generically a $S_l$-covering. By Lemma \ref{lem:small_bc}, $\mathrm{id}_T \times f'$ is small and generically a $S_{l}$ covering. Since the morphism $\gr$ is smooth and thus flat (see the proof of \cite[Proposition 11]{Heinloth_LaumonBDay}), we deduce by Lemma \ref{lem:small_bc} that $f$ is small and generically a $S_l$-covering. By Theorem \ref{thm:actions}, there is an induced $S_l$-action on $M(\THecke^l_{\cE/T})$ and we can now prove the following result.

\begin{thm}\label{thm Sl action on THecke}
Let $\cE$ be family of rank $n$ degree $d$ vector bundles over $C$ parametrised by a smooth $k$-scheme $T$. Then via the isomorphism $M(\THecke^l_{\cE/T}) \cong M(T) \otimes M(C \times \PP^{n-1})^{\otimes l}$ of Corollary \ref{cor motive THecke}, the $S_l$-action permutes the $l$-copies of $M(C \times \PP^{n-1})$. Moreover, we have
\[ M(\Hecke^l_{\cE/T}) \cong M(T) \otimes M(\Sym^l(C \times \PP^{n-1})). \]
\end{thm}
\begin{proof}
We note that as $T$ is smooth, both $\Hecke^l_{\cE/T}$ and $\THecke^l_{\cE/T}$ are smooth over $k$. By Lemma \ref{lemma PB for THecke with line bundles}, there is an isomorphism
\[ M(\THecke^l_{\cE/T}) \cong M(T) \otimes M(C \times \PP^{n-1})^{\otimes l} \]
induced by $l$ line bundles $\cL_1^l, \dots, \cL_l^l$ on $ \THecke^l_{\cE/T}$ (which are the pullbacks of the ample bundles on each projective bundle) and the projection $P_l : \THecke^l_{\cE/T} \ra T \times C^l$. The $S_l$-action on $M (\THecke^l_{\cE/T} )$ from Theorem \ref{thm:actions} is induced by the $S_l$-action on the open subset $\THecke^{l,\circ}_{\cE/T} = p^{-1} ( \Hecke^{l,\circ}_{\cE/T})$, where $\Hecke^{l,\circ}_{\cE/T}$ parametrises length $l$ Hecke correspondences whose degree $l$ torsion quotient has support consisting of $l$ distinct points. The $S_l$-action on $\THecke^{l,\circ}_{\cE/T} $ corresponds to permuting the $l$ universal degree $1$ torsion quotients $\cT_1^l, \dots, \cT_l^l$. By Remark \ref{rmk line bundles THecke}, this corresponds to permuting the $l$ line bundles $\cL_i^l$ on $\THecke^l_{\cE/T}$ (see equation \eqref{eq line bundles and torsion quotients}). Therefore, the induced $S_l$-action on $M(\THecke^l_{\cE/T})$ permutes the $l$-copies of $M(C \times \PP^{n-1})$. As $f$ is a small proper surjective map of smooth varieties, Theorem \ref{thm:actions} yields an isomorphism
\[ M (\THecke^l_{\cE/T} )^{S_l} \cong M( \Hecke^l_{\cE/T} ).\]
Finally, by Example \ref{ex:sym} we have $\Sym^{nl-d}M(C \times \PP^{n-1})\simeq M(\Sym^{nl-d}(C\times \PP^{n-1}))$.
\end{proof}

In particular, if we apply this to $T = \Spec k$ and $\cE = \cO_C(D)^{\oplus n}$ for a divisor $D$ on $C$, we obtain Theorem \ref{thm intro FDiv} as a special case of this result. Furthermore, the motive of the Quot scheme of length $l$ torsion quotients of a locally free sheaf $\cE$ over $T \times C/T$ only depends on the rank of $\cE$; we explicitly state this as a corollary for $T = \spec k$, as there are similar recent results concerning the class of such Quot schemes in the Grothendieck ring of varieties \cite{BFP, Ricolfi}.

\begin{cor}\label{cor motive torsion quot}
Let $\cE$ be rank $n$ locally free sheaf on $C$ and $l \in \NN$. Then the motive of the Quot scheme $ \quot_{C/k}^{(0,l)}(\cE)$ parametrising length $l$ torsion quotients of $\cE$ is
\[ M(\quot_{C/k}^{(0,l)}(\cE)) \cong  M(\Sym^l(C \times \PP^{n-1})).\]
In particular, this motive only depends on the rank $n$ of $\cE$.
\end{cor}

\section{The formula for the motive of the stack of vector bundles}\label{sec formula}

\subsection{The transition maps in the inductive system}\label{sec lift trans}

Throughout this section we fix $x \in C(k)$ and let $s_x : \Spec k \ra C$ be the inclusion of $x$. The inclusion $\cO_C \hookrightarrow \cO_C(x)$ defines an inductive sequence of morphisms $i_l :  \Div_{n,d}(l) \ra \Div_{n,d}(l+1)$ indexed by $l \in \NN$. In this section, we will lift the maps $i_l :  \Div_{n,d}(l) \ra \Div_{n,d}(l+1)$ to the schemes of iterated Hecke correspondences and compute the induced maps of motives. We recall that 
\[ \Div_{n,d}(l) = \Hecke^{nl -d}_{\cO_C(lx)^{\oplus n}/ \Spec k} \quad \text{and} \quad \FDiv_{n,d}(l) = \THecke^{nl -d}_{\cO_C(lx)^{\oplus n}/ \Spec k}\]
and we will drop the subscripts for Hecke schemes  throughout the rest of this section.

The inclusion $\cO_C \hookrightarrow \cO_C(x)$ induces an inclusion $\cO_C^{\oplus n} \hookrightarrow \cO_C(x)^{\oplus n}$. Any full flag
\[ \cF_\bullet =(\cO_C^{\oplus n}=\cF_0 \subsetneq \cF_1 \subsetneq \cdots \subsetneq \cF_{n-1} \subsetneq \cF_n =  \cO_C(x)^{\oplus n})\]
determines, for $l \in \NN$, a morphism $A_l({\cF_\bullet}) : \FDiv_{n,d}(l)  \ra \FDiv_{n,d}(l+1) $ lifting the morphism $ \Div_{n,d}(l) \ra \Div_{n,d}(l+1)$. Recall that we have maps $P_{nl-d} : \FDiv_{n,d}(l) = \THecke^{nl-d} \ra C^{nl-d}$ defined in Remark \ref{rmk structure of Hecke schemes}. The morphism $A_l({\cF_\bullet})$ sits in a commutative diagram
 \begin{equation}\label{commutative diag A a}
\xymatrixcolsep{3pc} \xymatrix{ \FDiv_{n,d}(l)  \ar[d]_{P_{nl-d}} \ar[r]^-{A_l({\cF_\bullet})} &\FDiv_{n,d}(l+1)  \ar[d]^{P_{n(l+1)-d}}\\
C^{nl-d} \ar[r]^-{c_l} &C^{n(l+1) -d}. }
 \end{equation}
where $c_l:= s_{x}^n \times \mathrm{id}_{C^{nl-d}} $. Recall that $\pr_j^{nl-d} : \FDiv_{n,d}(l) \ra C^{nl-d} \ra C$ denotes the composition of $P_{nl-d}$ with the projection onto the $j$th factor. We have
\begin{equation}\label{eq A_l and projections}
\pr_j^{n(l+1)-d} \circ A_l({\cF_\bullet}) = \left \{ \begin{array}{ll} t_x & \text{if } 1 \leq j \leq n\\ \pr^{nl-d}_{j-n }& \text{if } n+1 \leq j \leq n(l+1)-d, \end{array} \right. 
\end{equation}
where $t_x : \FDiv_{n,d}(l) \ra \spec k \ra C$ is the composition of the structure map with $s_x$. 

Similarly, a tuple $p:=(p_1, \cdots , p_n) \in (\PP^{n-1})^n$ induces $b_l(p) : (\PP^{n-1})^{nl-d} \ra (\PP^{n-1})^{n(l+1)-d}$ which is the identity on the last $nl-d$ factors. We define
\[ a_l(p):=c_l \times b_l(p) : (C \times \PP^{n-1})^{nl-d} \ra (C \times \PP^{n-1})^{n(l+1)-d}.\]

\begin{lemma}
Every choice of flag $\cF_\bullet$ induces the same map of motives
\[ M(A_l):= M(A_l({\cF_\bullet})) : M(\FDiv_{n,d}(l)) \ra M( \FDiv_{n,d}(l+1))\]
and every choice of tuple $p \in (\PP^{n-1})^n$ induces the same map of motives
\[ M(b_l)=M(b_l(p)) : M(\PP^{n-1})^{\otimes nl-d} \ra M(\PP^{n-1})^{\otimes n(l+1)-d}.\]
\end{lemma}
\begin{proof}
A flag $\cF_\bullet$ as above is specified by a full flag in $k^n$, which is parametrised by the flag variety $\GL_n/B$, which is $\AA^1$-chain connected and so all flags induce the same map of motives. The second statement follows similarly as projective spaces are also $\AA^1$-chain connected.
\end{proof}

As we are only interested in studying these maps motivically, we will drop the choice of flag $\cF_\bullet$ and tuple $p$ from the notation and simply write $A_l$, $b_l$ and $a_l$ for these morphisms.

By Lemma \ref{lemma PB for THecke with line bundles}, there is an $S_{nl-d}$-equivariant isomorphism \[\PB(\cL^{nl-d}_\bullet): M(\FDiv_{n,d}(l)) = M( \THecke^{nl-d}) \ra M(C \times \PP^{n-1})^{\otimes nl-d}\]
determined by line bundles $\cL_j^{nl-d}$  on $\THecke^{nl-d}$ for $1 \leq j \leq nl-d$. Moreover, we have homomorphisms $\varphi_l : S_{nl-d} \hookrightarrow S_{n(l+1)-d}$ such that the maps $c_l : C^{nl-d} \ra C^{n(l+1)-d}$ are equivariant. 

\begin{prop}\label{prop THecke trans}
For each $l$, we have a commutative diagram
\begin{equation}\label{eq prop THecke tran}
\xymatrixcolsep{3pc}
 \xymatrix{ M( \FDiv_{n,d}(l)) \ar[d]_{\PB(\cL^{nl-d}_\bullet)}^{\wr} \ar[r]^-{M(A_l)} &M(  \FDiv_{n,d}(l+1) )\ar[d]^{\PB(\cL^{n(l+1)-d}_\bullet)}_{\wr}\\
M(C \times \PP^{n-1})^{\otimes nl-d} \ar[r]^-{M(a_l )} &M(C \times \PP^{n-1})^{\otimes n(l+1) -d} }
\end{equation}
such that the horizontal maps are equivariant with respect to $\varphi_l : S_{nl-d}  \hookrightarrow S_{n(l+1)-d}$. 
\end{prop}
\begin{proof} 
We claim that the pullbacks via $A_l$ (for any flag $\cF_\bullet$) of the line bundles $\cL_j^{n(l+1)-d}$ satisfy
\begin{equation}\label{eq pullback line bundles under A}
 A_l^* \cL^{n(l+1)-d}_j = \left \{ \begin{array}{ll} \cO_{\THecke^{nl-d}} & \text{if } 1 \leq j \leq n \\ \cL^{nl-d}_{j-n} & \text{if } n+1 \leq j \leq n(l+1)-d.  \end{array} \right.
\end{equation}
We recall that we have $n(l+1)-d$ families of degree $1$ torsion sheaves on $C$ parametrised by $\FDiv_{n,d}(l+1)=\THecke^{n(l+1)-d}$ given by the successive quotients of the universal flag of vector bundles on $\THecke^{n(l+1)-d} \times C$; these families of torsion sheaves are denoted by
\[ \cT^{n(l+1)-d}_j:= \cU^{n(l+1)-d}_{j-1}/\cU^{n(l+1)-d}_{j}  \quad \text{ for } \: 1 \leq j \leq  n(l+1)-d .\]
The pullbacks of these families of torsion sheaves along $A_l$ (for any flag $\cF_\bullet$) are as follows:
\begin{equation}\label{eq pullback torsion under A}
 (A_l \times \mathrm{id}_C)^* \cT^{n(l+1)-d}_j = \left \{ \begin{array}{ll} p_C^*k_x & \text{if } 1 \leq j \leq n \\ \cT^{nl-d}_{j-n} & \text{if } n+1 \leq j \leq n(l+1)-d, \end{array} \right.
\end{equation}
where $p_C : \THecke^{n(l+1)-d} \times C \ra C$ denote the projection and $k_x$ is the skyscraper sheaf at $x$. Consequently, Claim \eqref{eq pullback line bundles under A} follows from equations \eqref{eq line bundles and torsion quotients}, \eqref{eq A_l and projections} and \eqref{eq pullback torsion under A}.

Similarly, if we let  $\cM_j^{nl-d}$ denote the line bundle on $(C \times \PP^{n-1})^{nl-d}$ obtained by pulling back $\cO_{\PP^{n-1}}(1)$ via the $j$th projection, we have
\begin{equation*}
 a_l^* \cM^{n(l+1)-d}_j = \left \{ \begin{array}{ll} \cO_{(C \times \PP^{n-1})^{nl-d}} & \text{if } 1 \leq j \leq n \\ \cM^{nl-d}_{j-n} & \text{if } n+1 \leq j \leq n(l+1)-d.  \end{array} \right.
\end{equation*}
Since the action of the symmetric groups on these motives corresponds to permuting the order of these line bundles, we see that $M(A_l)$ and $M(a_l)$ are both equivariant with respect to $\varphi_l$.

Finally let us prove the commutativity of the square \eqref{eq prop THecke tran}. For this we require the explicit formula for the iterated projective bundle isomorphisms given in Lemma \ref{lemma PB for THecke with line bundles}:
\[ \PB(\cL_\bullet^{nl-d}) =(M(P_{nl-d}) \otimes [c_1(\cL_\bullet^{nl-d})]) \circ M(\Delta_{\THecke^{nl-d}}),\]
where $[c_1(\cL_\bullet^{nl-d})] : M(\THecke^{nl-d}) \ra M(\PP^{n-1})^{nl-d}$ is the map induced by powers of the first chern classes of the line bundles $\cL_j^{nl-d}$ for $ 1 \leq j \leq nl-d$. If we insert $n$ copies of the structure sheaf on $\THecke^{nl-d}$ into this family, we obtain a map
\[ [c_1(\cO, \dots , \cO,\cL_\bullet^{nl-d})] : M(\THecke^{nl-d}) \ra M(\PP^{n-1})^{n(l+1)-d}. \]
In fact, since $c_1(\cO)$ is the zero map, we see that $[c_1(\cO)] : M(\THecke^{nl-d}) \ra M(\PP^{n-1})$ is the composition of the structure map $M(\THecke^{nl-d}) \ra \QQ\{ 0 \}$ with the inclusion $\QQ\{ 0 \} \hookrightarrow M(\PP^{n-1})$ of any point in $\PP^{n-1}$. Therefore, we can write the lower diagonal composition in \eqref{eq prop THecke tran} as
\[ M(a_l) \circ \PB(\cL_\bullet^{nl-d}) = (M(c_l \circ P_{nl-d}) \otimes [c_1(\cO, \dots , \cO, \cL_\bullet^{nl-d})] )\circ M(\Delta_{\THecke^{nl-d}}). \]
Then by \eqref{eq pullback line bundles under A}, we have 
\[ [c_1(\cO, \dots , \cO, \cL_\bullet^{nl-d})] =  [c_1(\cL_\bullet^{n(l+1)-d})] \circ M(A_l)\]
and as diagram \eqref{commutative diag A a} commutes, we deduce that
\[ \PB(\cL_\bullet^{n(l+1)-d}) \circ M(A_l) =  (M(c_l \circ P_{nl-d}) \otimes [c_1(\cO, \dots , \cO, \cL_\bullet^{nl-d}) ])\circ M(\Delta_{\THecke^{nl-d}}), \]
which completes the proof that the square \eqref{eq prop THecke tran} commutes.
\end{proof}

Since $a_l : (C \times \PP^{n-1})^{nl-d} \ra (C \times \PP^{n-1})^{n(l+1)-d}$ is equivariant with respect to $\varphi_l : S_{nl-d}  \hookrightarrow S_{n(l+1) -d}$, we obtain an induced map between the associated symmetric products
\[ \xymatrixcolsep{5pc} \xymatrix{   (C \times \PP^{n-1})^{nl-d} \ar[r]^{a_l} \ar[d] & (C \times \PP^{n-1})^{n(l+1)-d} \ar[d] \\ 
\Sym^{nl-d}(C \times \PP^{n-1}) \ar[r]^-{\Sym(a_l)} & \Sym^{n(l+1)-d}(C \times \PP^{n-1}).}\]
By Theorem \ref{thm intro FDiv}, there is an isomorphism
\[ e_l : M(\Div_{n,d}(l)) \cong M(\FDiv_{n,d}(l))^{S_{nl-d}} \cong \Sym^{nl-d}M(C \times \PP^{n-1}) \]
where the second isomorphism is induced by the $S_{nl-d}$-equivariant isomorphism $\PB(\cL^{nl-d}_\bullet)$.

\begin{cor}\label{cor1 transition maps div}
The following diagram commutes
\[ \xymatrixcolsep{5pc} \xymatrix{  M(\Div_{n,d}(l))  \ar[r]^{M(i_l)} \ar[d]_{e_l}^{\wr} & M(\Div_{n,d}(l+1)) \ar[d]^{e_{l+1}}_{\wr} \\ 
\Sym^{nl-d}M(C \times \PP^{n-1}) \ar[r]^-{M(\Sym(a_l))} & \Sym^{n(l+1)-d}M(C \times \PP^{n-1}).}\]
\end{cor}
\begin{proof}
By the equivariance property of $M(A_{l})$ observed in Proposition \ref{prop THecke trans} and the fact that $A_l$ lifts $i_l$, the isomorphisms of Theorem \ref{thm intro FDiv} fit in a commutative diagram
  \[
\xymatrixcolsep{5pc}  \xymatrix{
    M(\FDiv_{n,d}(l))^{S_{nl-d}} \ar[r]^{M(A_l)} \ar[d]_{\wr} & M(\FDiv_{n,d}(l+1))^{S_{n(l+1}-d} \ar[d]_{\wr} \\
M(\Div_{n,d}(l)) \ar[r]^{M(i_l)} & M(\Div_{n,d}(l+1))
  }
\]
The corollary then follows from combining this diagram with the diagram of Proposition \ref{prop THecke trans}.
\end{proof}

\subsection{A proof of the formula}\label{sec proof1}

The rational point $x \in C(k)$ gives rise to a decomposition $M(C) = \QQ\{ 0 \} \oplus \overline{M}(C)$, where $\overline{M}(C) = M_1(\Jac(C)) \oplus \QQ\{ 1\}$, see \cite[Proposition 4.2.5]{AHEW}. The motive of $\Jac(C)$ can be recovered from the motive $ M_1(\Jac(C))$ using \cite[Proposition 4.3.5]{AHEW}:
\[ M(\Jac(C)) = \bigoplus_{i=0}^{2g} \Sym^i(M_1(\Jac(C))) = \bigoplus_{i=0}^{\infty} \Sym^i(M_1(\Jac(C))) . \]
We can then write
\[ M(C \times \PP^{n-1})  = M(C) \otimes \left(\bigoplus_{i=0}^{n-1} \QQ\{i \}\right) = \QQ\{ 0 \} \oplus \overline{M}(C) \oplus   \bigoplus_{i=1}^{n-1} M(C) \{i \}.\]
Let $M_{C,n}:= \overline{M}(C) \oplus  \bigoplus_{i=1}^{n-1}M(C)\{i \}$; then (for example, by \cite[Lemma B.3.1]{AHEW})
\[ \Sym^{nl-d}(M(C \times \PP^{n-1})) =  \Sym^{nl-d}(\QQ\{ 0 \} \oplus M_{C,n})= \bigoplus_{i=0}^{nl-d} \Sym^i(M_{C,n}). \]

\begin{lemma}\label{lemma tran map Div using MC}
There is a commutative diagram
\[ \xymatrixcolsep{5pc} \xymatrix{  M(\Div_{n,d}(l))  \ar[r]^{M(i_l)} \ar[d]^{\wr} & M(\Div_{n,d}(l+1)) \ar[d]_{\wr} \\ 
\bigoplus\limits_{i=0}^{nl-d} \Sym^i(M_{C,n}) \ar[r] & \bigoplus\limits_{i=0}^{n(l+1)-d} \Sym^i(M_{C,n})}\]
where the lower map is the obvious inclusion.
\end{lemma}
\begin{proof}
Let us start with the description of the transition map given in Corollary \ref{cor1 transition maps div}. We see that the map $a_l : (C \times \PP^{n-1})^{nl-d} \ra (C \times \PP^{n-1})^{n(l+1)-d}$ can be described motivically as
\[ M(a_l) : M(C \times \PP^{n-1})^{\otimes nl-d} \cong \QQ\{ 0 \}^{\otimes n} \otimes M(C \times \PP^{n-1})^{\otimes (nl-d)}  \stackrel{\iota^{\otimes n} \otimes M(\mathrm{id})}{\xrightarrow{\hspace*{1cm}}}   M(C \times \PP^{n-1})^{\otimes n(l+1)-d} \]
where $\iota : \QQ \{0 \} \ra M(C \times \PP^{n-1}) = \QQ\{0 \} \oplus M_{C,n}$ is the natural inclusion of this direct factor. It thus follows that the symmetrised map $M(\Sym(a_l))$ is the claimed inclusion.
\end{proof}

\begin{thm}
If $C(k) \neq \emptyset$, then the motive of $\Bun_{n,d}$ satisfies
\[ M(\Bun_{n,d}) \simeq \hocolim_{l} \left( \bigoplus_{i=0}^{nl-d}  \Sym^i(M_{C,n}) \right) \simeq \bigoplus_{i=0}^{\infty} \Sym^i(M_{C,n}).\]
More precisely, we have
\[M(\Bun_{n,d}) \simeq  M(\Jac(C)) \otimes M(B\GG_m) \otimes \bigotimes_{i=1}^{n-1} Z(C, \QQ\{i\}). \]
\end{thm}
\begin{proof}
The first claim follows from Lemma \ref{lemma tran map Div using MC} and Theorem \ref{thm old main}. For the second claim, we introduce the notation  $\Sym^*(M):= \oplus_{i=0}^{\infty} \Sym^i(M)$ for any motive $M$; then 
\begin{enumerate}[label={\upshape(\roman*)}]
\item $\Sym^*(M_1 \oplus M_2) = \Sym^*(M_1) \otimes \Sym^*(M_2)$ (by \cite[Lemma B.3.1]{AHEW}),
\item $Z(C,\QQ\{i\}) = \Sym^*(M(C)\{i \})$ (by definition of the motivic Zeta function), 
\item $\Sym^*(\QQ\{1\})= M(B\GG_m)$ (see \cite[Example 2.21]{HPL} based on  \cite[Lemma 8.7]{totaro}),
\item $\Sym^*(M_1(\Jac(C))) = M(\Jac(C))$ (by \cite[Proposition 4.3.5]{AHEW}),
\end{enumerate}
and the formula follows from these observations.
\end{proof}

\subsection{An alternative proof using previous results}\label{sec proof2}

We will give a second proof of this formula for $M(\Bun_{n,d})$, also based on Corollary \ref{cor1 transition maps div} but which follows more closely our previous work \cite{HPL}. The idea is to describe the unsymmetrised transition maps $M(a_l)$ by decomposing the motives $M(C \times \PP^{n-1})^{\otimes nl-d}$ using $M(\PP^{n-1}) = \oplus_{i=0}^{n-1} \QQ\{ i \}$.

\begin{rmk}
By returning to the decomposition $M(\PP^{n-1}) = \oplus_{i=0}^{n-1} \QQ\{ i \}$, we can describe the maps $M(a_l)$ explicitly. Indeed we have a decomposition $M(C \times \PP^{n-1})^{\otimes nl-d}$ indexed by ordered tuples $I = (i_1, \cdots, i_{nl-d}) \in  \cI_l:= \{ 0, \cdots ,n-1\}^{\times \: nl -d }$ of the form
\[ M(C \times \PP^{n-1})^{\otimes nl-d} = \bigoplus_{I \in \cI_l} \bigotimes_{j=1}^{nl-d}  M(C) \{i_j\}  = \bigoplus_{{I \in \cI_l}} M(C^{nl+d}) \{ | I |  \},\]
where $| I | = \sum_{j=1}^{nl-d} i_j$. 

There is a map $h_l: \cI_l \ra \cI_{l+1}$ given by $I \mapsto (0,\dots, 0,I)$ (inserting $n$ zeros) such that the map $M(a_l): M(C \times \PP^{n-1})^{\otimes nl-d} \ra M(C \times \PP^{n-1})^{\otimes n(l+1)-d}$ sends the direct summand indexed by $I \in \cI_l$ to the direct summand indexed by the tuple $h_l(I) \in \cI_{l+1}$ via the map
\begin{equation}\label{eq beh unsym trans}
M(c_l)\{| I |\} : M(C^{nl+d})\{ | I | \} \ra M(C^{n(l+1)+d})\{ | I | \} = M(C^{n(l+1)+d})\{ | (0,\dots, 0,I) | \}. 
\end{equation}
\end{rmk}

The $S_{nl-d}$-action on $M(C \times \PP^{n-1})^{\otimes nl-d}$ permutes these direct summands via the obvious action of $S_{nl-d}$ on $\cI_l$. The invariant part is the motive of $\Sym^{nl-d}(C \times \PP^{n-1})$ which has an associated decomposition. The index set for this decomposition is
\[ \cB_l:=\left\{ m =  (m_0, \dots , m_{n-1}) \in \NN^n : \sum_{i=0}^{n-1} m_i = nl-d \right\}. \]
Moreover, for $I \in \cI_l$, we let $\tau_l(I)_r= \# \{ i_j : i_j = r \}$, then $\tau_l(I) = (\tau_l(I)_0, \dots , \tau_l(I)_{n-1}) \in \cB_l$ and the map $\tau_l : \cI_l \ra \cB_l$ is $S_{nl-d}$-invariant with $| I | = \sum_{i=0}^{n-1} i \tau_l(I)_i$. By grouping together the factors with the same values of $i_j$, there is a map
\begin{equation}\label{eq map induced by s}
C^{nl+d} \ra  \prod_{i=0}^{n-1} \Sym^{\tau_l(I)_i}(C) 
\end{equation}
which is the quotient of the natural action of $\Stab(I) \cong \prod_{i=0}^{n-1} S_{\tau_l(I)_i}$.

\begin{lemma}\label{lemma sym transition maps}
For each $l$, we have a decomposition
\[ M( \Sym^{nl-d}(C \times \PP^{n-1})) = \bigoplus_{m \in \cB_l} \bigotimes_{i=0}^{n-1} \Sym^{m_i}(M(C)) \{i m_i \}\]
such that the following statements hold.
\begin{enumerate}[label={\upshape(\roman*)}]
\item \label{Sym decomp} For each $m \in \cB_l$, we have a commutative diagram
\[ \xymatrix{  M(C \times \PP^{n-1})^{\otimes nl-d} \ar[r] \ar[d] & M( \Sym^{nl-d}(C \times \PP^{n-1}))  \ar[d] \\
\bigoplus\limits_{I \in \tau_l^{-1}(m)} M(C^{nl-d}) \{ | I |  \}  \ar[r] & \bigotimes\limits_{i=0}^{n-1} \Sym^{m_i}(M(C)) \{i m_i \} } \]
where the lower maps are induced by the maps \eqref{eq map induced by s}.
\item \label{Sym trans} The transition maps $M(\Sym(a_l))$ decompose as maps 
\[ \kappa_{m,m'} : \bigotimes_{i=0}^{n-1} \Sym^{m_i}(M(C)) \{i m_i \} \ra \bigotimes_{i=0}^{n-1} \Sym^{m'_i}(M(C)) \{i m'_i \} \]
for $m \in \cB_l$ and $m' \in \cB_{l+1}$ with $\kappa_{m,m'} = 0$ unless $m' = m + (n,0,\dots 0)$, in which case this map is induced by the morphism of varieties
\[ \prod_{i=0}^{n-1}\Sym^{m_i}(C)\ra  \prod_{i=0}^{n-1}\Sym^{m'_i}(C) \]
which is the map $\Sym(s_x^n \times \mathrm{id}_{C^{m_i}})$ on the $0$th factor and the identity on all other factors. 
\end{enumerate}
\end{lemma}
\begin{proof}
We will give the decomposition and the proof of \ref{Sym decomp} simultaneously, by collecting the direct summands in the decomposition of $M(C \times \PP^{n-1})^{\otimes nl-d}$ which are preserved by the $S_{nl-d}$-action and taking their invariant parts. For this, we recall that there is a $S_{nl-d}$-action on $\cI_l$ and the map $\tau_l : \cI_l \ra \cB_l$ is $S_{nl-d}$-invariant and the fibres consist of single orbits. For $I \in \cI_l$ with $m = \tau_l(I)$, we note that the quotient of the associated action of $\Stab(I) = \prod_{i=0}^{n-1} S_{m_i}$ on $C^{nl-d}$ is isomorphic to $\prod_{i=0}^{n-1} \Sym^{m_i}(C)$. Therefore, the motive appearing in the left lower corner of the diagram in statement \ref{Sym decomp} is a direct summand of $M(C \times \PP^{n-1})^{\otimes nl-d}$ that is preserved by the $S_{nl-d}$-action and its $S_{nl-d}$-invariant piece is precisely the motive appearing in the lower right corner. This proves the first statement and the decomposition. 

To describe the behaviour of the symmetrised transition maps with respect to this decomposition, we recall that the unsymmetrised transition maps send the direct summand indexed by $I \in \cI_l$ to $h_l(I)=(0,...0,I) \in \cI_{l+1}$. The unsymmetrised transition maps on these direct summands are described by \eqref{eq beh unsym trans} and so it remains to describe the induced map on the invariant parts for the actions of the symmetric groups. Since $h_l : \cI_l \ra \cI_{l+1}$ is equivariant for the actions of the symmetric groups via the homomorphism $\varphi_l : S_{nl-d} \hookrightarrow S_{n(l+1)-d}$, it descends to map
\[ \overline{h} : \cB_l \ra \cB_{l+1} \quad \text{where} \quad \overline{h}(m) = m + (n,0, \dots, 0). \]
Thus, $\kappa_{m,m'}$ is zero unless $m' = \overline{h}(m)$. For $m' = \overline{h}(m)$, $I \in \tau_l^{-1}(m)$ and $I' \in \tau_{l+1}^{-1}(m')$ note that
\[ | I | = | I'| = \sum_{i=0}^{n-1} i m_i = \sum_{i=0}^{n-1} i m_i' \]
and
\[ \Stab(I) = \prod_{i=0}^{n-1} S_{m_i} \quad \text{and} \quad \Stab(I') = \prod_{i=0}^{n-1} S_{m'_i} = S_{m_0 + n} \times \prod_{i=1}^{n-1} S_{m_i}. \] 
In particular, the map $c_l = s_x^n \times \mathrm{id} : C^{nl-d} \ra C^{n(l+1)-d}$ is equivariant for the induced actions of $\Stab(I)$ and $\Stab(I')$ and there is a map between the quotients
\[ \xymatrix{  C^{nl-d}  \ar[r]^{c_l} \ar[d] &  C^{n(l+1)-d}  \ar[d] \\
\prod\limits_{i=0}^{n-1} \Sym^{m_i}(C)  \ar[r] & \prod\limits_{i=0}^{n-1} \Sym^{m'_i}(C)   } \]
which is $\Sym(s_x^n \times \mathrm{id}_{C^{m_i}})$ on the $0$th factor and the identity on the other factors. Combined with \ref{Sym decomp}, this concludes the proof of \ref{Sym trans}.
\end{proof}

\begin{cor}\label{cor2 transition maps div}
The transition maps $M(i_l) : M(\Div_{n,d}(l)) \ra M(\Div_{n,d}(l+1))$ fit in the following commutative diagram
\[ \xymatrixcolsep{5pc} \xymatrix{  M(\Div_{n,d}(l))  \ar[r]^{M(i_l)} \ar[d]^{\wr} & M(\Div_{n,d}(l+1)) \ar[d]_{\wr} \\ 
\bigoplus\limits_{m \in \cB_l} \bigotimes\limits_{i=0}^{n-1} \Sym^{m_i}(M(C)) \{i m_i \} \ar[r]^-{\bigoplus\limits_{m,m' } \kappa_{m,m'}} & \bigoplus\limits_{m' \in \cB_{l+1}} \bigotimes\limits_{i=0}^{n-1} \Sym^{m'_i}(M(C)) \{i m'_i \}.}\]
where the maps $\kappa_{m,m'}$ are as in Lemma \ref{lemma sym transition maps}. 
\end{cor}
\begin{proof} 
This follows from Lemma \ref{lemma sym transition maps} and Corollary \ref{cor1 transition maps div}.
\end{proof}

This looks very similar to \cite[Conjecture 3.9]{HPL}, except we do not know whether the vertical maps in this commutative diagram coincide with the maps given by the Bia{\l}ynicki-Birula decompositions used in the formulation of this conjecture. Nevertheless, with the description of the transition maps in Corollary \ref{cor2 transition maps div}, one can apply the proof of \cite[Theorem 3.18]{HPL} to obtain an alternative proof of the formula for the motive of $\Bun_{n,d}$ appearing in Theorem \ref{main thm}.

\bibliographystyle{abbrv}
\bibliography{references}

\medskip \medskip

\noindent{Freie Universit\"{a}t Berlin, Arnimallee 3, 14195 Berlin, Germany} 

\medskip \noindent{\texttt{hoskins@math.fu-berlin.de, simon.pepin.lehalleur@gmail.com}}

\end{document}